\documentclass[12pt]{amsart}
\usepackage{amssymb,amsmath,amsthm,amssymb,amsfonts,amscd}
\usepackage{graphicx}
\usepackage{color, enumerate}
\usepackage[all]{xy}
\usepackage[english]{babel}
\usepackage{dsfont} 

\usepackage{cancel}

\newtheorem{theorem}{Theorem}[section]
\newtheorem{lemma}[theorem]{Lemma}
\newtheorem{definition}[theorem]{Definition}
\newtheorem{corollary}[theorem]{Corollary}
\newtheorem{proposition}[theorem]{Proposition}

\DeclareMathAlphabet{\mathpzc}{OT1}{pzc}{m}{it}

\newcommand{\N}{\mathbb{N}}

\def\eps{\varepsilon}

\def\U{\mathcal{U}}
\def\L{\mathcal{L}}

\def\X{{\rm \bf X}}

\def\Z{{\rm \bf Z}}
\def\x{{\rm \bf x}}

\def\a{{\rm \bf a}}

\def\j{{\rm \bf j}}

\def\B2star{\overline{B}_X^{w(X^{\ast\ast},X^{\ast})}}

\title{A Lindenstrauss theorem for some classes of multilinear mappings}
\author{Daniel Carando %
\and Silvia Lassalle %
\and Martin Mazzitelli}

\thanks{This project was supported in part by CONICET PIP 0624, PICT 2011-1456 and UBACyT 1-746}

\address{Departamento de Matem\'{a}tica - Pab I,
Facultad de Cs. Exactas y Naturales, Universidad de Buenos Aires,
(1428) Buenos Aires, Argentina and IMAS-CONICET}

\email{dcarando@dm.uba.ar,  mmazzite@dm.uba.ar}

\address{Departamento de Matem\'{a}tica, Universidad de San
Andr\'{e}s, Vito Dumas 284, (B1644BID) Victoria, Buenos Aires,
Argentina and IMAS - CONICET.}

\email{slassall@udesa.edu.ar}


\date{}

\begin{document}
\baselineskip=.65cm

\maketitle

\centerline{\small{{\emph {Dedicated to Professor Richard Aron on the occasion of his 70th birthday}}} }

\begin{abstract}
Under some natural hypotheses, we show that if a multilinear mapping belongs to some Banach
multlinear ideal, then it can be approximated by multilinear mappings \emph{belonging to the same ideal}
all whose Arens extensions attain their norms at the same point. We prove a similar result for the class of symmetric
multilinear mappings. We see that the quantitative (Bollob\'as-type) version of these results fails in every multilinear ideal.

\end{abstract}

\section*{Introduction}
The Bishop-Phelps theorem \cite{BishPhel61,BishPhel63} is an elementary and significant result about
continuous linear functionals and convex sets. The most quoted version asserts that the set of
linear functionals in $X'$ (the dual space of a Banach space $X$) which attain their supremum on
the unit ball of $X$ is norm-dense in $X'$. Lindenstrauss showed that this is not true, in general,
for linear bounded operators between two Banach spaces $X$ and $Y$~\cite{Lind}; while he
proved that the set of bounded linear operators whose second adjoints attain their norm, is always
dense in the space of all bounded operators. This result was later extended for multilinear
mappings by Acosta, Garc\'{\i}a and Maestre \cite{AGM06}. These kinds of results are referred to as
Lindenstrauss-type theorems.

Regarding multilinear mappings, the Bishop-Phelps theorem fails, in general, even for scalar-valued
bilinear forms \cite{AAP, Choi97}. In order to handle the study of Lindenstrauss-type multilinear
results, the Arens extensions come into scene \cite{Arens}. The first result in this setting was
given by Acosta \cite{Aco98} who proved that the set of bilinear forms on a product of two
Banach spaces $X$ and $Y$ such that their Arens extensions are norm attaining is dense
in the space of bilinear forms $\L(^2X\times Y)$.  Aron, Garc\'{\i}a and Maestre \cite{ArGM}
obtained an improvement by showing that the set of those mappings in $\L(^2 X\times Y)$ such that
the two possible Arens extensions attain the norm at the same element of $X\times Y$,  is dense in
$\L(^2 X\times Y)$. This last result is stronger than the previous one since there exist  bilinear
mappings such that only one of their Arens extensions attains the norm \cite{ArGM}. In  \cite{AGM06},
the authors prove that the strongest version holds with full generality for multilinear mappings.
They also give several positive results of the kind for some multilinear ideals considering  the
ideal norm instead of the supremum norm.

In this paper, we show that a multilinear Lindenstrauss theorem holds for any ideal of 3-linear
forms and, as a consequence, we obtain the same result for any regular ideal of bilinear mappings. More generally, we
prove in Theorem \ref{Un estable} that an $N$-linear Lindenstrauss theorem holds for any  $N\in \mathbb N$, for a wide class of
multilinear ideals which preserve some algebraic structure related to multiplicativity (we say these ideals are stable). Our
results include the classes of nuclear, integral, extendible, multiple $p$-summing mappings ($1\le
p <\infty$). Also, if we consider multilinear mappings on Hilbert spaces, the class of
Hilbert-Schmidt and, more generally, the multilinear Schatten classes are encompassed.  It was observed in \cite{AGM06} that any ideal which is dual to an associative tensor norm satisfies the multilinear Lindenstrauss theorem. These ideals are easily seen to be stable, but the converse is not true: the ideal of multiple 2-summing mappings is stable, and we show  in Proposition~\ref{noesasociativa} that it cannot be dual to any associative tensor norm. Our list of
examples then extends and completes the multilinear ideals treated in~\cite{AGM06}.

In \cite{CaLaMa12} we give an integral representation formula for the duality between
tensor products and polynomials on Banach spaces satisfying appropriate hypotheses. Here, we extend
\cite[Theorem~2.2]{CaLaMa12} for the duality between tensor products and multilinear mappings.  As a consequence, we obtain in Theorem \ref{Lind simetrico} a  Lindenstrauss-type theorem for the space of symmetric multilinear
mappings $\mathcal{L}_s(^NX;Z)$, whenever $X$ has separable dual with the approximation property and $Z$ is a dual space or a Banach space with the property $(\beta)$ of Lindenstrauss. We also provide examples of Banach spaces for which a Bishop-Phelps-type theorem fails, but our Lindenstrauss theorem holds.

Finally, we address quantitative versions of the Bishop-Phelps and Lindenstrauss theorems in ideals of multilinear mappings. In \cite{Boll}, Bollob\'as improved the theorem of Bishop and Phelps showing that it is possible not only to approximate linear functionals by norm-attaining ones, but also to `somehow' choose  the elements where the norm is attained. In the last years, a lot of attention has been paid to Bishop-Phelps-Bollob\'as-type results in the linear, multilinear and polynomial context. We show that the corresponding quantitative version of the Lindenstrauss theorem fails for every ideal of multilinear mappings. On the other hand, we show the Bishop-Phelps-Bollob\'as theorem for any ideal of multilinear mappings defined on a product of uniformly convex Banach spaces. This extends, to the ideal setting, some results obtained in \cite{Acoet6, ABGM, KimLee}.

\section{Lindenstrauss theorem in multilinear ideals} \label{Lindenstrauss en ideales}

Let us fix some notation. Throughout this paper $X$ and $Y$ denote Banach spaces, while $X'$ and
$B_X$ denote respectively the topological dual and the closed unit ball of $X$. For Banach spaces
$X_1, \ldots, X_N$ we denote the product space $X_{1} \times \cdots \times X_{N}$ by
$\X$ and by $\L(^N\X;Y)$ the space of continuous $N$-linear
mappings $\Phi \colon \X\to Y$ endowed with the supremum norm. Recall that the {\it Arens
extensions of a multilinear function} are obtained by weak-star density. Each extension depends on
the order  in
which the variables are extended.  Here we present one of the $N!$ possible extensions (see
\cite{Arens} and \cite[1.8]{DefFlo93}).
Given $\Phi \in \L(^N\X;Y)$, the mapping $\overline{\Phi} : X''_1
\times\cdots\times X''_N \longrightarrow Y''$ is defined by
\begin{eqnarray}\label{arens extension}
\overline{\Phi}(x_1'',\ldots,x_N'') = w^* - \lim_{\alpha_1}\ldots\lim_{\alpha_N}
\Phi(x_{1,{\alpha_1}},\ldots,x_{N,{\alpha_N}})
\end{eqnarray}
where $(x_{j,{\alpha_j}})_{\alpha_j } \subseteq X$ is a (bounded) net  $w^*$-convergent to  $x''_j\in X''_j$,
$j=1,\ldots, N$. For $N=1$ this recovers the definition of the bitranspose of a continuous
operator.
Now, we recall the definition of a multilinear ideal.

\begin{definition}
A normed ideal of $N$-linear mappings is a pair $\left( \U , \Vert
\cdot \Vert_{\U}\right)$ such that for any $N$-tuple of Banach
spaces $\X=X_1\times\cdots\times X_N$ satisfies
\begin{enumerate}[\upshape (i)]

\item $\U(\X ; Y) = \U \cap \mathcal{L}(^N \X;Y)$ is a linear subspace of
$\mathcal{L}(^N \X;Y)$  for any Banach space $Y$ and $\Vert \cdot
\Vert_{\U}$ is a norm on it.

\item For any $N$-tuple of Banach spaces $\Z=Z_1\times\cdots\times
Z_N$, any Banach space $W$ and operators $T_i \in
\mathcal{L}(Z_i;X_i)$, $1 \leq i \leq N$, $S\in \mathcal{L}(Y;W)$
and $\Phi \in \U(\X; Y)$, the $N$-linear mapping $S \circ \Phi \circ (T_1,
\dots, T_n)\colon \Z\longrightarrow W$ given by
$$
S \circ \Phi \circ (T_1, \dots, T_N)(z_1, \dots, z_N) = S(\Phi(T_1(z_1),
\dots, T_n(z_N)))
$$
belongs to $\U(\Z; W)$ with $\Vert S \circ \Phi \circ (T_1, \dots, T_N)
\Vert_{\U} \leq \| S\|  \| \Phi\|_{\U} \| T_1\| \dots \| T_N\|.$

\item $(z_1, \dots, z_N) \mapsto z_1 \cdots z_N$ belongs to $\U( \mathbb{C}^N ; \mathbb{C})$ and has
norm one.
\end{enumerate}
If  $\left( \U(\X,Y) , \Vert \cdot \Vert_{\U}\right)$ is complete for every $\X$ and $Y$,
we say that  $\left( \U , \Vert \cdot \Vert_{\U}\right)$ is a Banach
ideal of $N$-linear mappings. In the scalar-valued case, we simply write $\U(\X)$.
\end{definition}

For $N \in \mathbb N$ and
$\j=\{ j_1, \ldots, j_p \}$ a subset of the initial set
$\{1,\ldots, N\}$, $j_1<j_2 <\cdots < j_p$, we define $P_\j\colon
\X\to \X$ the projection given by
\begin{equation*}
P_\j(x_{1}, \dots, x_{N}) := (y_1, \dots, y_N), \quad \text{where}
\quad y_k = \left \{
\begin{array} {lcl}
x_k & \text{if} & k \in  \j\\
0 & \text{if} & k \notin \j
\end{array}
\right..
\end{equation*}
If $\j^c$ denotes the complement of $\j$ in $\{1, \dots,
N\}$, then $P_\j + P_{\j^c}=Id$, the identity  map on $\X$.

Let us define a rather natural property for multilinear ideals which will ensure the validity of a
Lindenstrauss-type theorem. Let us say that the ideal of $N$-linear forms $\U$ is {\it stable
 at $\X$}
if there exists $K>0$ such that for all  $\a=(a_1, \ldots, a_N)\in\X$ and all $\j\subset \{1,\ldots,
N\}$ , the function $V_{\j, \a}\colon \L(^N\X) \to \L(^N\X)$
defined by
\begin{equation} \label{definicion de Vj}
V_{\j, \a}(\Phi)(\x) = \Phi(P_\j(\x) + P_{\j^c}(\a)) \Phi(P_\j(\a) +
P_{\j^c}(\x))
\end{equation}
satisfies
\begin{equation} \label{propiedad de Vj}
V_{\j, \a}(\Phi) \in \U(\X)\quad  \text{for all}\  \Phi \in \U(\X)\quad
\text{and} \quad \| V_{\j,\a}(\Phi)\|_{\U} \leq K \| \Phi\|_{\U}^2 \|a_1\|\cdots\|a_N\|.
\end{equation}

In order to see that being stable is a natural property, take $N=4$ and $\j=\{1,2\}$. In this case,
what we are imposing to a 4-linear form $\Phi\in \U(\X)$ is that the mapping
$$(x_1,x_2,x_3,x_4)\mapsto \Phi(x_1,x_2,a_3,a_4)\,\Phi(a_1,a_2,x_3,x_4)$$ also belongs to $\U(\X)$
for any $(a_1,\dots, a_4)$, with some control on the norm.
{The next result extends \cite[Theorem~2.1]{AGM06} and
\cite[Corollary~2.5]{AGM06}.}

\begin{theorem}\label{Un estable}
If the ideal of $N$-linear forms $\U$ is stable at
$\X=X_1\times \dots\times X_N$, then the set of
$N$-linear forms in $\U(\X)$ whose Arens extensions
attain the supremum-norm at the same $N$-tuple is $\|
\cdot\|_{\U}$-dense in $\U(\X)$.
\end{theorem}

\begin{proof} Fix $\Phi\in\L(^N \X)$. By the proof of \cite[Theorem~2.1]{AGM06}, there exists
a sequence of multilinear mappings $(\Phi_n)_n$ given recursively by
$$
\Phi_1 = \Phi, \quad \Phi_{n+1} = \Phi_n + \sum_{\j} C_n V_{\j,\a^n}(\Phi_n),
$$
where $(C_n)_n$ is a sequence of positive numbers, $(\a^n)_n\subset
\X$ and $V_{\j,\a^n}$ is defined as in \eqref{definicion de Vj} for
all $\j\subset \{1,\ldots,N\}$. In \cite[Theorem~2.1]{AGM06} it is
shown that, given $\eps >0$, $(C_n)_n$ and $\a^n$ can be chosen so
that $(\Phi_n)_n$ converges to an element $\Psi\in\L(^N\X)$ whose Arens
extensions attain their norm at the same $N$-tuple and $\| \Phi - \Psi\| <
\varepsilon$.

The stability of $\U$ implies that every $\Phi_n$ belongs
to $\U(\X)$ whenever $\Phi\in \U(\X)$. Now, the control of the norms  given in \eqref{propiedad de
Vj}
together with a careful reading of
the proof of~\cite[Corollary~2.5]{AGM06} ensure that $(C_n)_n$ and
$\a^n$ can be chosen so that $(\Phi_n)_n$ converges to an element $\Psi=
\| \cdot\|_{\U} - \lim \Phi_n$ and $\|\Phi - \Psi\|_{\U} < \varepsilon$.
\end{proof}

Given an $(N+1)$-linear form $ \Phi\colon X_1\times\dots\times
X_N\times X_{N+1} \to \mathbb K$, we
define
the associate $N$-linear mapping  $\tilde \Phi\colon  X_1\times\dots\times X_N\to X_{N+1}'$ as
usual:
$$\tilde \Phi (x_1,\dots,x_N)(x_{N+1})=\Phi (x_1,\dots,x_N,x_{N+1}).$$ Now, given an ideal of
$N$-linear mappings $\U$ we  define the ideal of $(N+1)$-linear forms $\widetilde \U$ by $$\Phi\in
\widetilde \U \quad\text{if and only if}\quad \tilde\Phi\in \U$$ and $$\|\Phi\|_{\widetilde {\U}}\colon =\|\tilde \Phi\|_{\U}.$$

Recall that a Banach ideal $\U$ is called \textit{regular} if $J_Y \circ \Phi \in \U(\X;Y'')$ implies $\Phi \in \U(\X;Y)$ and $\|\Phi\|_{\U}=\|J_Y \circ \Phi\|_{\U}$, where $J_Y \colon Y \hookrightarrow Y''$ is the natural injection.

It should be noted that  if $\Phi$ in the proof of the previous theorem is $w^*$-continuous in the
last variable,  then so is $\Psi$. As a consequence, we can proceed as in \cite[Theorem~2.3]{AGM06}
to obtain the
following.

\begin{corollary}
With the notation above, if $\U$ is regular and $\widetilde \U$ is stable at $X_1\times \dots\times X_N\times Y'$, then
the set of $N$-linear mappings in $\U(\X;Y)$ such that their Arens
extensions attain the supremum-norm at the same $N$-tuple is $\|
\cdot\|_{\U}$-dense in $\U(\X;Y)$.
\end{corollary}

\medskip
We will see that most of the known examples of multilinear ideals are stable and, then, satisfy a
Lindenstrauss-type theorem. First, let us see that this property is fulfilled for \emph{every} ideal of
3-linear forms. Hence, we have a Lindenstrauss theorem for ideals of 3-linear forms and regular ideals of
bilinear mappings.

\begin{corollary}
Let $\left( \U , \| \cdot \|_{\U}\right)$ be a Banach ideal of
3-linear forms. Then, for every $\X=X_1\times X_2\times X_3$, the set of 3-linear forms in $\U(\X)$
whose Arens extensions attain the supremum-norm at the same
$3$-tuple is $\| \cdot\|_{\U}$-dense in $\U(\X)$.
\end{corollary}

\begin{proof}
By Theorem~\ref{Un estable}, it suffices to prove that $\U(\X)$ is
stable. Take $\Phi \in \U(\X)$, $\a =(a_1,a_2,a_3) \in \X$ with
$\|a_k\|=1$ for $1 \leq k \leq 3$ and take $\j \subset \{1,2,3\}$.
We proceed to show that~\eqref{propiedad de Vj} is satisfied for
$\j=\{1,2\}$, the other cases are analogous.

Consider the linear operator $T_\Phi\colon X_3 \rightarrow X_3$ defined
by $T_\Phi(x_3) = \Phi(a_1,a_2,x_3)a_3$. Then  $\|T_\Phi\| \leq \| \Phi \| \leq \| \Phi \|_{\U}$ and
$$
V_{\j, \a}(\Phi)(\x)= \Phi(x_1,x_2, a_3) \Phi(a_1,a_2,x_3) = \Phi(x_1,x_2,
T_\Phi(x_3)) = \Phi \circ (I,I,T_\Phi) (x_1,x_2,x_3).
$$
Since $(\U, \| \cdot\|_{\U})$ is a  Banach ideal,
$V_{\j, \a}(\Phi) \in
\U(\X)$ and $\|V_{\j,\a}(\Phi)\|_{\U} \leq \| \Phi \|_{\U}^2$.
\end{proof}

Now the following corollary is immediate.

\begin{corollary}
{Let $\left( \U , \| \cdot \|_{\U}\right)$ be a regular
Banach ideal of bilinear mappings and let $X_1, X_2, Y$ be
Banach spaces. The set of  bilinear mappings of $\U(X_1\times X_2; Y)$
attaining their norm at the same pair is $\| \cdot\|_{\U}$-dense in
$\U(X_1\times X_2; Y)$.}
\end{corollary}

\section{Examples of multilinear ideals satisfying a Lindenstrauss theorem}

Let us start with the simplest examples. Any finite type multilinear form on $X_1\times\cdots\times
X_N$ ($i.e.$ any linear combinations of products of linear forms),  has a unique Arens extension to
$X_1''\times\cdots\times X_N''$ which is weak-star continuous on each coordinate.
By the Banach-Alaogu's theorem, this extension attains its supremum norm. As a consequence,
every Banach ideal in which finite type multilinear forms are dense satisfies the multilinear
Lindenstrauss theorem. This is the case, for instance, of the ideal of nuclear multilinear
forms. More generally, if $\U$ is a minimal ideal of multilinear forms, the finite type
multilinear forms are dense in $\U$ \cite{Flo01,FlGa03} and Lindenstrauss theorem trivially
holds.

As observed in \cite{AGM06}, if $\U$ is an ideal of multilinear forms which is dual to an
associative tensor norm (such as the injective or projective tensor norms $\varepsilon$ and $\pi$),
then $\U$ satisfies the multilinear Lindenstrauss theorem. We can rephrase their remark in our
terminology: {\it ideals  which are dual to associative tensor norms are always stable}. In
\cite{AGM06}, it is mentioned that the ideal of multiple summing multilinear mappings satisfies the
Lindenstrauss theorem. We will see that this is truly the case, although this ideal is not dual to
any associative tensor norm as we show in Proposition~\ref{noesasociativa}. Actually, it is not very
usual for ideals of multilinear forms to be dual to associative tensor norms. Fortunately, in order
to satisfy a multilinear Lindenstrauss theorem (in fact, in order to be stable), a much weaker
property is sufficient. Coherent and multiplicative ideals of polynomials have been studied (also
with different terminologies) in \cite{CaDiMu09,BoPe05}. Here, we present a multilinear version of
these properties (see \cite{BoCaPe11}, where similar properties for multilinear mappings are
considered). To simplify the definitions, we restrict to symmetric ideals, although this
is clearly not necessary.

Fix $N \in \mathbb N$ and $\X=X_{1} \times
\cdots \times X_{N}$.
If $\theta$ is a permutation of $\{1,\ldots, N\}$, we write $$\X_\theta=X_{\theta(1)} \times \cdots
\times X_{\theta(N)}$$ and $\x_\theta=(x_{\theta(1)}, \ldots,
x_{\theta(N)})$.
We say that the ideal of multilinear mappings $\U_N$ is \textbf{symmetric} if for any $\Phi \in
\U_N(\X;Y)$ and every permutation $\theta$ of $\{1, \dots, N\}$, the $N$-linear
mapping $\theta \Phi\colon \X_{\theta} \longrightarrow Y$,
$$\theta \Phi (\x_\theta) = \Phi(\x)
$$
belongs to $ \U_N(\X_\theta;Y)$ with $\|\theta
\Phi\|_{\U_{N}} = \| \Phi\|_{\U_{N}}$.

\begin{definition}
Let $\mathfrak U=(\U_n)_{n}$ be a sequence where $\U_n$ is a symmetric Banach ideal
of $n$-linear forms for each $n \in \mathbb N$. We say that $\mathfrak U$ is
\textbf{multiplicative} if there exist positive constants $C$ and $D$ such
that, for any $N\in \mathbb N$ and any $\X=X_1\times \cdots\times X_N$:
\begin{enumerate}[\upshape (i)]
\item For $\Phi \in \U_N(\X)$ and $a_N \in X_N$, the $(N-1)$-linear form  $\Phi_{a_N}$ given by
$$
\Phi_{a_N}(x_1, \dots, x_{N-1}) = \Phi(x_1, \dots, x_{N-1}, a_N),
$$
belongs to $ \U_{N-1}(X_1\times \dots\times X_{N-1};Y)$ and  $\|
\Phi_{a_N}\|_{\U_{N-1}} \leq C \| \Phi\|_{\U_N} \| a_N\|$.

\item For $\Phi\in\U_{k}(X_1\times \dots\times X_k)$ and $\Psi\in\U_{N-k}(X_{k+1}\times \dots\times
X_N)$,
the $N$-linear form $\Phi\cdot \Psi$ given by
$$(\Phi\cdot \Psi) (x_1, \dots, x_N) =  \Phi(x_1, \dots,
x_k) \Psi(x_{k+1},\dots,x_N)
$$
belongs to $\U_N(\X)$ and $\| \Phi\cdot \Psi\|_{\U_{N}}
\leq D^{N} \| \Phi\|_{\U_{k}} \| \Psi\|_{\U_{N-k}}$.
\end{enumerate}
\end{definition}

It is rather  easy to see that if $\mathfrak U=(\U_n)_{n}$ is a multiplicative
sequence of multilinear ideals, then $\U_n$ is stable at any Banach space, for any $n\in \mathbb N$.
What makes this concept interesting in our framework is that most of the usual ideals of multilinear
forms have been proven to be multiplicative. For example, the ideals of nuclear, integral,
extendible, multiple $p$-summing multilinear forms ($1\le p <\infty$) are multiplicative (see
\cite{CaDiMu07,CaDiMu12,tesisMuro}  for the proof in the polynomial case, the multilinear one being
analogous). Then, the Lindenstrauss theorem holds for all these ideals. Also, if we consider multilinear forms on Hilbert spaces, the class of Hilbert-Schmidt
and, more generally, the multilinear Schatten classes are multiplicative.
As a consequence, since Hilbert spaces
are reflexive, these ideals satisfy a multilinear Bishop-Phelps
theorem.

We end this section by showing that the ideal of multiple summing mappings is not dual to any
associative tensor norm.

\begin{definition}
Let  $1 \leq p < \infty$. A multilinear form $\Phi : X_1 \times \dots \times X_N \rightarrow \mathbb
K$ is  \textit{multiple $p$-summing} if there exists $K >0$ such that for any sequences
$(x_{i_j}^j)_{i_j = 1}^{m_j} \subseteq X_j$, $j=1,\ldots, N$, we have
\begin{equation} \label{definicion multiple p sumante}
\left( \sum_{i_1, \dots,i_N=1}^{m_1, \dots,m_N} | \Phi(x_{i_1}^1, \dots, x_{i_N}^N)|^p\right)^{1/p}
\leq K \prod_{j=1}^N \Vert (x_{i_j}^j)_{i_j = 1}^{m_j}\Vert_{p}^w.
\end{equation}
The least constant $K$ satisfying the inequality is the  $p$-summing norm of $T$ and is denoted by
$\pi_{p}(T)$. We write $\Pi_{p}^N(X_1\times\dots \times X_N)$ for the space of multiple $p$-summing
forms.
\end{definition}
Recall that $\Pi_{p}^N(X_1\times\dots\times X_N)$ is the dual of the tensor product
$X_1\otimes\cdots\otimes X_N$ endowed with the tensor norm $\tilde{\alpha}_p(u)$ \cite[Proposition
3.1]{P-GV03} (see also \cite{Ma03}) where
\begin{equation}
\tilde{\alpha}_p(u) = \inf \{ \sum_{m=1}^M \Vert (\lambda_{m,i_m^1, \dots,i_m^N})_{i_m^1,
\dots,i_m^N =1}^{I_m^1, \dots, I_m^N}\Vert_{p'} . \Vert (x_{m,i_m^1}^1)_{i_m^1=1}^{I_m^1}\Vert_p^w
\dots \Vert (x_{m,i_m^N}^1)_{i_m^N=1}^{I_m^N}\Vert_p^w \}
\end{equation}
with $\frac{1}{p} + \frac{1}{p'} =1$ and the infimum is taken over all the representations of the
form $$u = \sum_{m=1}^{M} \sum_{i_m^1, \dots, i_m^N=1}^{I_m^1, \dots, I_M^N} \lambda_{m,
i_m^1,\dots,i_m^N} \, x_{m,i_m^1}^1 \otimes \dots \otimes x_{m,i_m^N}^N .$$

Although we know that the $(\Pi_{2}^n)_n$ is a multiplicative sequence, we have the following.

\begin{proposition}\label{noesasociativa}
The ideal $\Pi_{2}^N$ ($N\in \mathbb N$) of multiple 2-summing forms is not dual to any associative
tensor norm.
\end{proposition}

\begin{proof}
Take $N=4$ and  $X_1=\cdots =X_4= c_0$. Suppose that $\beta$ is  an associative tensor norm of order
2 such that
$$
\Pi_{2}^4(c_0\times\cdots\times c_0) \simeq \Big( \big( c_0 \tilde{\otimes}_{\beta} c_0\big)
\tilde{\otimes}_{\beta}  \big( c_0 \tilde{\otimes}_{\beta} c_0\big)\Big)'.
$$
By  \cite[Theorem 3.1]{Bombal-PerezGarcia-Villanueva}, every  multilinear  form on $c_0$ is multiple
2-summing. As a consequence, the projective tensor norm $\pi$ and the tensor norm predual to the multiple 2-summing forms
should be equivalent on  $c_0\otimes\cdots\otimes c_0$. Using this fact first for the
4-fold and then for the 2-fold tensor products, we have
$$
\big( c_0 \tilde{\otimes}_{\pi} c_0\big) \tilde{\otimes}_{\pi}  \big( c_0 \tilde{\otimes}_{\pi}
c_0\big) \simeq\big( c_0 \tilde{\otimes}_{\beta} c_0\big) \tilde{\otimes}_{\beta}  \big( c_0
\tilde{\otimes}_{\beta} c_0\big) \simeq  \big( c_0 \tilde{\otimes}_{\pi} c_0\big)
\tilde{\otimes}_{\beta}  \big( c_0 \tilde{\otimes}_{\pi} c_0\big).
$$

In \cite{Cabello-Garcia-Villanueva} it is shown that $c_0 \tilde{\otimes}_{\pi} c_0$ has uniformly
complemented copies of $\ell_2^n$. Then, the isomorphisms given above imply that
$$
\ell_2^n \otimes_\pi \ell_2^n \simeq \ell_2^n \otimes_\beta \ell_2^n,
$$
uniformly in $n \in \mathbb N$. The Density Lemma \cite[13.4]{DefFlo93} then gives $$\ell_2
\otimes_\pi \ell_2 \simeq \ell_2 \otimes_\beta \ell_2, $$
which means that every bilinear form on $\ell_2\times\ell_2$ is multiple 2-summing. But in Hilbert
spaces, multiple 2-summing  and  Hilbert-Schmidt multilinear forms coincide, and clearly there are
bilinear forms which are not Hilbert-Schmidt. This contradiction completes the proof.
\end{proof}

\section{Integral representation and Lindenstrauss theorem for symmetric multilinear mappings}
\label{formula integral} \label{lindenstrauss polinomial}

We devote this section to the special class of symmetric multilinear mappings. Since symmetric $N$-linear mappings are defined on $\X$ for $X_1=\cdots=
X_N=X$, we simply write $\mathcal{L}_s(^NX;Z)$ to denote the space of these mappings (with values
on a Banach space $Z$). We prove a  Lindenstrauss theorem for $\mathcal{L}_s(^NX;Z)$ under certain hypotheses on $X$ and $Z$. Recall that a
Banach space $Z$ has property $(\beta)$ of Lindenstrauss, see  \cite{Lind}, if there exists a subset $\{(z_\alpha, g_\alpha): \, \alpha \in \Lambda\} \subset Z \times Z'$ satisfying:
\begin{enumerate}[\upshape (i)]
\item $\Vert z_\alpha\Vert = \Vert g_\alpha\Vert = g_\alpha(z_\alpha) =1$,
\item $\vert g_\alpha(z_\gamma)\vert \leq \lambda$ for $\alpha \neq \gamma$ and some $0\leq \lambda < 1$,
\item for all $z \in Z$, $\Vert z\Vert = \sup_{\alpha \in \Lambda} \vert g_\alpha(z)\vert$.
\end{enumerate}
Examples of spaces with this property are $c_0$, $\ell_\infty$ and $C(K)$ with $K$ having a dense set of isolated points.

Our main result in this section is the following.

\begin{theorem} \label{Lind simetrico}
Let $X$ be a Banach space whose dual is separable and
has the approximation property and let $Z$ be a dual space or a Banach space with property $(\beta)$. Then, every  symmetric multilinear mapping in $\mathcal{L}_s(^NX;Z)$ can be approximated by
symmetric multilinear mappings whose Arens extensions
attain the supremum-norm at the same $N$-tuple.
\end{theorem}

The proof of Theorem~\ref{Lind simetrico} is based on the following lemmas. The first one extends to the multilinear setting the integral formula for the duality between tensor products and polynomials given in  \cite[Theorem~2.2]{CaLaMa12}. We briefly sketch the corresponding proof. The second lemma extends \cite[Theorem~2.1~(ii)]{ChoiKim96}, stated under the framework of Bishop-Phelps-type theorems. We omit its proof which follows the one given in \cite{ChoiKim96} (see also\cite{CaMa} where Lemma~\ref{de escalar a beta} is obtained in the polynomial context).

\begin{lemma}\label{formula integral} Let $\X=X_1\times \cdots\times X_N$  be an $N$-tuple of Banach spaces each of which has   separable dual with approximation property and let $Y$ be a Banach space. Then,
 for each $u \in (\tilde{\otimes}_{\pi,
j=1,\dots,N}^{N}X_j) \tilde{\otimes}_{\pi} Y$, there exists a regular
Borel measure $\mu_u$ on $(B_{X_1''}, w^*) \times \cdots \times
(B_{X_N''}, w^*) \times (B_{Y''}, w^*)$ such that $\Vert \mu_u\Vert
\leq \Vert u\Vert_{\pi}$ and
\begin{equation} \label{ec. formula integral vectorial}
\left\langle u,\Psi \right\rangle = \int_{B_{X_1''}\times \cdots \times
B_{X_N''} \times B_{Y''}} \overline{\Psi}(x_1'', \dots, x_N'')(y'')
d\mu_u(x_1'',\dots,x_N'',y''),
\end{equation}
for all $\Psi \in \mathcal{L}(^N\X;Y')$, where $\pi$ is the projective tensor norm.
\end{lemma}

\begin{proof} Since for each $j=1,\dots,N$, the space $X_j$ has separable dual with approximation property, a combination of Proposition~3.5 and Theorem~3.10 of \cite{HandbookI} ensures that
there exists a bounded sequence of finite
rank operators $(T^j_n)_n$ on  $X_j$ such that both  $T^j_n\longrightarrow
Id_{X_j}$ and $(T^j_n)'\longrightarrow Id_{X_j'}$ in the strong operator
topology.

Now, given $\Psi \in \mathcal{L}(^N\X;Y')$, for each $(n_1,\dots,n_N)\in\N^N$
we define the finite type multilinear mapping
$$
\Psi_{n_1,\ldots,n_N} = \Psi \circ(T^1_{n_1},\ldots,T^N_{n_N}),
$$
and, proceeding as in \cite[Lemma 2.1]{CaLaMa12}, we see that the Arens extension of $\Psi$ is given
by
\begin{equation} \label{limite de medibles}
\overline{\Psi}(x_1'', \dots, x_N'')(y'') = \lim_{n_1 \rightarrow
\infty} \ldots \lim_{n_N \rightarrow \infty}
\overline{\Psi_{n_1,\ldots,n_N}}(x_1'', \dots, x_N'')(y'').
\end{equation}
By the Riesz representation theorem for
$C\big( (B_{X_1''},w^*)\times \cdots \times (B_{X_N''},w^*) \times (B_{Y''},w^*)\big)$, there is
a regular Borel measure $\mu$ satisfying \eqref{ec. formula integral vectorial} for finite type multilinear mappings.
Then, by \eqref{limite de medibles}, the Dominated convergence theorem and the density of
linear combination of elementary tensors we obtain  \eqref{ec. formula integral vectorial}
for every continuous $N$-linear mapping.
\end{proof}

\begin{lemma}\label{de escalar a beta}
Suppose that $Z$ has property $(\beta)$. If the Lindenstrauss theorem holds for $\mathcal{L}_s(^NX)$, then it also holds for $\mathcal{L}_s(^NX;Z)$.
\end{lemma}

We denote by $\tilde{\otimes}_{\pi}^{N,s}X$ the $N$-fold symmetric tensor product of $X$ endowed with the (full, not symmetric) projective tensor norm $\pi$. We refer to \cite{Flo97} for general theory on
symmetric tensor products.

\begin{proof}[Proof of Theorem~\ref{Lind simetrico}] Suppose first that $Z$ is a dual space, say $Z=Y'$. Take $\Phi \in \mathcal{L}_s(^NX;Y')$
 and $\varepsilon>0$, and
consider the associated linear functional $$L_{\Phi} \in \left(
(\tilde{\otimes}_{\pi}^{N,s}X) \tilde{\otimes}_{\pi} Y
\right)'.$$ By the Bishop-Phelps theorem, there exists  a norm attaining
functional $L=L_{\Psi}$ such that $\Vert \Phi - \Psi \Vert=\Vert L_{\Phi} - L_{\Psi} \Vert <
\varepsilon$, for some $\Psi \in \mathcal{L}_s(^NX;Y')$.
Take $u \in  (\tilde{\otimes}_{\pi}^{N,s}X)
\tilde{\otimes}_{\pi} Y $ with $\Vert u\Vert_{\pi}=1$ such that
$\vert L_{\Psi}(u)\vert = \Vert L_{\Psi}\Vert = \Vert \Psi\Vert$. By Lemma~\ref{formula integral}, there is a regular Borel measure
$\mu_{u}$ satisfying~\eqref{ec. formula integral vectorial}
and hence,
\begin{eqnarray*}
\Vert \Psi\Vert = \vert L_{\Psi}(u)\vert \leq \int_{B_{X''}\times
\cdots \times B_{X''} \times B_{Y''}} \vert\overline{\Psi}(x_1'', \dots,
x_N'')(y'')\vert d\vert\mu_u\vert(x_1'',\dots,x_N'',y'') \leq \Vert
\overline{\Psi}\Vert \Vert \mu_{u}\Vert \leq \Vert \Psi\Vert.
\end{eqnarray*}
As a consequence,
$\vert\overline{\Psi}(x_1'',\dots,x_N'')(y'')\vert = \Vert
\Psi\Vert$ almost everywhere (for $\mu_{u}$) and
$\overline{\Psi}$ attains its norm.
Note that changing the order of the iterated limits in \eqref{limite de medibles}, we obtain the $N!$ Arens extensions of $\Psi$. Then, any of these extensions attains its norm almost everywhere for $\mu_u$. In particular, there exists a $N$-tuple on which all the Arens extensions of $\Psi$ attain their norms simultaneously.

What we have just proved, implies the Lindenstrauss theorem in the scalar-valued case. Then, Lemma~\ref{de escalar a beta} gives the result for $Z$ with property $(\beta)$.
\end{proof}

The following proposition gives the converse of \cite[Theorem~2.1]{ChoiKim96}. For our purposes, we only state and prove the result concerning symmetric multilinear mappings. Nevertheless, with almost identical proof, the result remains valid for (non-necessarily symmetric) multilinear mappings defined on any $N$-tuple of Banach spaces $\X=X_1\times \cdots\times X_N$ and for homogeneous polynomials.

\begin{proposition} \label{equivalencia BP beta}
The Bishop-Phelps theorem holds for $\mathcal{L}_s(^NX)$ if and only if it holds for $\mathcal{L}_s(^NX;Z)$ for every (or some) Banach space $Z$ with property $(\beta)$.
\end{proposition}

\begin{proof}
Thanks to \cite[Theorem~2.1~(ii)]{ChoiKim96}, we only have to prove one implication.
Suppose $Z$ has property $(\beta)$ and take $\{(z_\alpha, g_\alpha): \, \alpha \in \Lambda\} \subset Z \times Z'$ and $0 \leq \lambda <1$ satisfying its definition. Let $\lambda < \lambda_0 < 1$, $\varepsilon < \lambda_0 - \lambda$ and fix $\varphi \in \mathcal{L}_s(^NX)$ with $\Vert \varphi\Vert =1$. Pick any $\alpha_0 \in \Lambda$ and consider
$$
\Phi(x_1, \dots, x_N) = \varphi(x_1, \dots, x_N) z_{\alpha_0} \in \mathcal{L}_s(^NX;Z).
$$
By hypothesis there exists a $\Psi \in \mathcal{L}_s(^NX;Z)$ with $\Vert \Psi\Vert =1$, attaining its norm at some $(a_1, \dots, a_N)$ and such that $\Vert \Psi - \Phi \Vert < \varepsilon$. Then, $\Vert g_{\alpha} \circ \Psi - g_{\alpha} \circ \Phi \Vert < \varepsilon$ for all $\alpha \in \Lambda$ and consequently,
\begin{equation} \label{desigualdad galpha}
\Vert g_{\alpha} \circ \Psi \Vert \leq \varepsilon  + \Vert g_{\alpha} \circ \Phi \Vert \leq \varepsilon + \lambda < \lambda_0 \quad \text{for every $\alpha \neq \alpha_0$.}
\end{equation}
Since $$1 = \Vert \Psi(a_1, \dots, a_N)\Vert = \sup_{\alpha \in \Lambda} \vert g_{\alpha} (\Psi (a_1, \dots, a_N))\vert,$$ it follows from \eqref{desigualdad galpha} that $\vert g_{\alpha_0} \circ \Psi (a_1, \dots, a_N) \vert = \Vert g_{\alpha_0} \circ \Psi \Vert = 1$. Hence, $g_{\alpha_0} \circ \Psi$ is norm attaining and $\Vert g_{\alpha_0} \circ \Psi - \varphi \Vert < \varepsilon$. This gives the desired statement.
\end{proof}

We finish this section with some examples of spaces for which the Bishop-Phelps theorem fails, but our Lindenstrauss theorem holds. We appeal to the classical preduals of Lorentz sequence spaces recalling only their definitions. For further details on these spaces and their applications in  norm attainment problems see \cite{JP, AAP, CaLaMa12}. By an \textit{admissible sequence} we mean a decreasing sequence $w =(w_i)_{i}$ of nonnegative real numbers such that $w_1=1$, $\lim_i w_i =0$ and $\sum_i w_i = \infty$. Given an admissible sequence $w$, the predual of the Lorentz sequence space $d(w,1)$, is the space $d_*(w,1)$ of all the sequences $x$ such that
$$
\lim_{n\to \infty}\frac{\sum_{i=1}^n x^*(i)}{W(n)}=0
$$
where $x^*=(x^*(i))_i$ is the decreasing rearrangement of $x =(x(i))_i$ and $W(n)= \sum_{i=1}^n w_i$. In this space the norm is defined by
$$
\Vert x\Vert_W : = \sup_{n}\frac{\sum_{i=1}^n x^*(i)}{W(n)} < \infty.
$$

For short, we denote by $NA\mathcal{L}_s(^Nd_*(w,1); Y)$ the set of norm attaining mappings in $\mathcal{L}_s(^Nd_*(w,1); Y)$.

\begin{proposition}
Let $w$ be an admissible sequence such that $w \in \ell_r$ for some $1 <r<\infty$ and let $Z$ be a Banach space with property $(\beta)$.
\begin{enumerate}[\upshape (i)]
\item $NA\mathcal{L}_s(^Nd_*(w,1))$ is not dense in $\mathcal{L}_s(^Nd_*(w,1))$ for $N \geq r$.
\item $NA\mathcal{L}_s(^Nd_*(w,1); Z)$ is not dense in $\mathcal{L}_s(^Nd_*(w,1); Z)$ for $N \geq r$.
\item $NA\mathcal{L}_s(^Nd_*(w,1); \ell_r)$ is not dense in $\mathcal{L}_s(^Nd_*(w,1); \ell_r)$ for every $N \in \mathbb{N}$.
\end{enumerate}
On the other hand, the Lindenstrauss theorem holds in the three cases above.
\end{proposition}

\begin{proof} To prove (i), take the symmetric $N$-linear mapping $\Phi(x_1, \dots, x_N) = \sum_{i=1}^\infty x_1(i)\cdots x_N(i)$ and proceed as in \cite[Theorem 2.6]{JP}.
For (iii), consider $\Phi(x_1, \dots, x_N) = (x_1(i) \cdots x_N(i))_{i}$  and reason again as in \cite[Theorem 2.6]{JP}, but using \cite[Lemma 4.2]{CaLaMa12} instead of Lemma 2.2 in there. Finally, (ii) follows from (i) and Proposition \ref{equivalencia BP beta}.

The last statement is a consequence of Theorem \ref{Lind simetrico}, since $d_*(w,1)$ has separable dual with the approximation property, $\ell_r$ is a dual space and $Z$ has property $(\beta)$.
\end{proof}

\section{On quantitative versions of  Bishop-Phelps and  Lindenstrauss theorems for ideals of multilinear mappings}

In \cite{Boll} Bollob\'as proved a quantitative version of the Bishop-Phelps theorem, known nowadays as the Bishop-Phelps-Bollob\'as theorem. Roughly speaking, this result states that for any Banach space $X$, given a linear functional $\varphi \in X'$ and $\tilde{x}  \in B_X$ such that $\varphi(\tilde{x})$ is close enough to $\Vert \varphi \Vert$,  it is possible to find a linear functional $\psi \in X'$ close to $\varphi$ attaining its norm at some $a \in B_X$ close to $\tilde x$.
For linear operators, this problem was first considered  by Acosta, Aron, Garc\'{\i}a and Maestre in
\cite{AAGM} where the following result is proved: the Bishop-Phelps-Bollob\'as theorem  holds  for $\mathcal L(\ell_1; Y)$ if and only if $Y$ has the so called \emph{approximate hyperplane series property}. Also,  the authors study a quantitative version of the Lindenstrauss theorem for operators, which will be referred to as a Lindenstrauss-Bollob\'as-type theorem. Unfortunately, even this question has in general a negative answer as \cite[Example~6.3]{AAGM} shows. The study of these type of problems for multilinear mappings is initiated by Choi and Song \cite{ChoiSong}. In contrast to the positive result for $\mathcal L(\ell_1; \ell_\infty)$ from \cite{AAGM}, it is shown in \cite{ChoiSong} that there is no Bishop-Phelps-Bollob\'as theorem for bilinear forms on $\ell_1 \times \ell_1$. However, some positive results were obtained in the multilinear and polynomial contexts. For instance, when $X_1, \dots, X_N$ are uniformly convex, the Bishop-Phelps-Bollob\'as theorem holds in $\mathcal{L}(^N\X;Y)$ for any Banach space $Y$  \cite{ABGM, KimLee}. An analogous result was proved in \cite{Acoet6} for the space of $N$-homogeneous polynomials.

Now we give  the definition of the Bishop-Phelps-Bollob\'as and Lindenstrauss-Bollob\'as properties for ideals of multilinear mappings.
We denote by $S_X$ and $S_\X$ the unit spheres of a Banach space $X$ and of the $N$-tuple $\X = X_1 \times \cdots \times X_N$, where $S_{\X} = S_{X_1} \times \cdots \times S_{X_N}$ is considered with the supremum norm. We also write $\X'$ instead of $X_1' \times \cdots \times X_N'$.
\smallskip

Let $\left( \U , \| \cdot \|_{\U}\right)$ be a Banach ideal of $N$-linear mappings and $X_1, \dots, X_N, Y$ be Banach spaces. We say that $\U(\X;Y)$ has the \textit{Bishop-Phelps-Bollob\'as property} ($BPBp$) if the following is satisfied: given $\varepsilon >0$ there exist $\beta(\varepsilon)$ and $\eta(\varepsilon)$ with $\lim_{\varepsilon \rightarrow 0^+} \beta(\varepsilon) =0$ such that, if $\Phi \in \U(\X; Y)$, $\Vert \Phi \Vert=1$ and $\tilde{\x} =(\tilde{x}_j)_{j=1}^N \in S_{\X}$ satisfy $\Vert \Phi(\tilde{\x}) \Vert > 1 - \eta(\varepsilon)$, then there exist $\Psi \in \U(\X;Y)$, $\| \Psi \|=1$, and $\a=(a_j)_{j=1}^N \in S_{\X}$ such that
$$
\| \Psi(\a) \| =1, \quad \|\a - \tilde{\x}\| < \beta(\varepsilon) \quad \text{and} \quad \|\Psi - \Phi\|_{\U}< \varepsilon.
$$

It is worth mentioning that definitions of this type appear for linear operators in \cite{ABGKM, ArCaKo} where the subclasses considered are (non-necessarily closed) subspaces of $\mathcal L$ under the supremum norm. Here, taking into account the results obtained in Section \ref{Lindenstrauss en ideales}, our definition requires approximation of the multilinear mappings in $\| \cdot \|_{\U}$.

Following \cite{AAGM, CaLaMa12} we say that $\U(\X;Y)$ has the \textit{Lindenstrauss-Bollob\'as property} ($LBp$) if, with $\varepsilon, \eta$ and $\beta$ as above, given $\Phi \in \U(\X; Y)$, $\Vert \Phi \Vert=1$ and $\tilde{\x} =(\tilde{x}_j)_{j=1}^N \in S_{\X}$ satisfying $\Vert \Phi(\tilde{\x}) \Vert > 1 - \eta(\varepsilon)$, there exist $\Psi \in \U(\X;Y)$, $\| \Psi \|=1$, and $\a''=(a_j'')_{j=1}^N \in S_{\X''}$ such that $$
\| \overline{\Psi}(\a'') \| =1, \quad \|\a'' - \tilde{\x}\| < \beta(\varepsilon) \quad \text{and} \quad \|\Psi - \Phi\|_{\U}< \varepsilon.
$$

In \cite{CaLaMa12}, provided that $w \in \ell_r$ for some $1 < r < \infty$, it is shown that the $LBp$ fails for $\mathcal{L}(^Nd_*(w,1))$ if $N \geq r$ and for $\mathcal{L}(^Nd_*(w,1); \ell_r)$ for every $N \in \mathbb N$. The given counterexamples are diagonal mappings, which do not necessary belong to any multilinear ideal; for instance, these mappings are not nuclear (nor approximable) since they are not weakly sequentially continuous. Our purpose now is to show counterexamples to the $LBp$ in every ideal $\U$ of multilinear mappings. The following results and the already mentioned counterexample to the $BPBp$ for bilinear forms, will be the key for our objective.
Recall that given a Banach space $X$, a linear projection $P: X'' \rightarrow X''$ is an $L$-projection if $$\Vert x''\Vert = \Vert P(x'')\Vert + \Vert x'' - P(x'')\Vert \quad \text{for all $x'' \in X''$,}$$ and $X$ is an $L$-summand in its bidual if it is the range of an $L$-projection. Examples of spaces that are $L$-summands in their biduals are $L_1(\mu)$-spaces, preduals of von Neumann algebras and Lorentz sequence spaces $d(w,1)$.

\begin{lemma}
Let $X, Y$ be Banach spaces such that $X$ is an $L$-summand in its bidual with $L$-projection $P$. Let $S: X'' \rightarrow Y''$ be a linear operator with $\|S\|=1$, $a'' \in S_{X''}$ and $\tilde{x} \in S_X$ satisfying $\|S(a'') \| =1$ and $\|a'' - \tilde{x}\| < \beta(\varepsilon) <1$ for some $\beta(\varepsilon) \xrightarrow[\varepsilon \rightarrow 0]{}\, 0$. Then, for $a = \frac{P(a'')}{\| P(a'') \|} \in S_X$ and some $\beta'(\varepsilon) \xrightarrow[\varepsilon \rightarrow 0]{}\, 0$ we have
 $$\|S(a) \| =1 \quad \text{and} \quad \|a - \tilde{x}\| < \beta'(\varepsilon).$$
\end{lemma}
\begin{proof}
Since $\beta(\varepsilon) > \|a'' - \tilde{x}\| = \| P(a'') - \tilde{x} \| + \| (I-P)(a'') \|$, it follows that $\| P(a'') - \tilde{x} \| < \beta(\varepsilon)$ and hence $\| P(a'') \| > 1 - \beta(\varepsilon) >0$. Then, we consider $a= \frac{P(a'')}{\| P(a'') \|} \in S_X$. Noting that $$1=\| S(a'') \| = \| S(P(a'')) + S((I-P)(a'')) \| \leq \| S(P(a''))\| + \| S((I-P)(a'')) \|$$ we obtain $$\| S(P(a''))\| \geq 1 - \| S((I-P)(a'')) \| \geq 1 - \| (I-P)(a'') \| = \| P(a'') \|$$ and consequently $\|S(a)\| \geq 1$, which gives $\|S(a)\| = 1$. Now, recalling that $\| P(a'') - \tilde{x} \| < \beta(\varepsilon)$ and $\| P(a'') \| > 1 - \beta(\varepsilon)$, we have
\begin{eqnarray*}
\|a - \tilde{x}\| &=& \frac{1}{\| P(a'') \|} \Big\| P(a'') - \| P(a'') \| \,  \tilde{x} \Big\| \\
&\leq& \frac{1}{\| P(a'') \|} \left( \| P(a'') - \tilde{x} \| +  \Big\| \tilde{x} - \|P(a'') \| \, \tilde{x} \Big\| \right) \\
&<&  \frac{1}{1 - \beta(\varepsilon)} \left( \beta(\varepsilon) +  1 - \|P(a'') \| \right)\\
&<& \frac{2 \beta(\varepsilon)}{1 - \beta(\varepsilon)} = \beta'(\varepsilon) \xrightarrow[\varepsilon \rightarrow 0]{}\, 0,
\end{eqnarray*}
which gives the desired statement.
\end{proof}

\begin{proposition}Let $\left( \U , \| \cdot \|_{\U}\right)$ be a Banach ideal of $N$-linear mappings, $Y$ be any Banach space and $X_1, \dots, X_N$ be $L$-summands in their biduals. If $\U(\X;Y)$ has the $LBp$ then it has the $BPBp$.
\end{proposition}
\begin{proof}
Call $P_1,\dots,P_N$ to the corresponding $L$-projections. Let $\varepsilon$, $\eta(\varepsilon)$ and $\beta(\varepsilon)$ be as in the definition of $LBp$, with $\varepsilon$ sufficiently small such that $\beta(\varepsilon) < 1$. Take $\Phi \in \U(\X; Y)$, $\Vert \Phi \Vert=1$ and $\tilde{\x} =(\tilde{x}_j)_{j=1}^N \in S_{\X}$ such that $\Vert \Phi(\tilde{\x}) \Vert > 1 - \eta(\varepsilon)$. By hypothesis, there exist $\Psi \in \U(\X;Y)$, $\| \Psi \|=1$, and $\a''=(a_j'')_{j=1}^N \in S_{\X''}$ satisfying $$\| \overline{\Psi}(\a'') \| =1, \quad \|\a'' - \tilde{\x}\| < \beta(\varepsilon) \quad \text{and} \quad \|\Psi - \Phi\|_{\U}< \varepsilon.$$
Consider $S_1: X_1'' \rightarrow Y''$ defined by $S_1(x_1'') = \Psi(x_1'', a_2'', \dots, a_N'')$. By the previous lemma $$\left\|S_1\left( \frac{P_1(a_1'')}{\| P_1(a_1'') \|} \right) \right\| =1 \quad \text{and} \quad \left\|\frac{P_1(a_1'')}{\| P_1(a_1'') \|} - \tilde{x}_1 \right\| < \beta'(\varepsilon)$$ for some $\beta'(\varepsilon) \xrightarrow[\varepsilon \rightarrow 0]{}\, 0$. Now taking $S_2: X_2'' \rightarrow Y''$ defined by $S_2(x_2'') = \Psi\left( \frac{P_1(a_1'')}{\| P_1(a_1'') \|}, x_2'', a_3'', \dots, a_N'' \right)$ and reasoning again with the previous lemma, we obtain $$\left\|\Psi\left( \frac{P_1(a_1'')}{\| P_1(a_1'') \|}, \frac{P_2(a_2'')}{\| P_2(a_2'') \|}, a_3'', \dots, a_N'' \right) \right\| =1 \quad \text{and} \quad \left\|\frac{P_i(a_i'')}{\| P_i(a_i'') \|} - \tilde{x}_i \right\| < \beta'(\varepsilon), \quad \text{$i=1,2$.}$$ Inductively, if we write  $\a=\left( \frac{P_1(a_1'')}{\| P_1(a_1'') \|}, \dots, \frac{P_N(a_N'')}{\| P_N(a_N'') \|} \right) $, we get
$$\left\Vert \Psi(\a) \right\Vert =1 \quad \text{and} \quad \left\Vert \a - \tilde{\x} \right\Vert < \beta'(\varepsilon) \xrightarrow[\varepsilon \rightarrow 0]{}\, 0 , $$ which gives the desired statement.
\end{proof}

Since $BPBp$ trivially implies $LBp$, the previous proposition gives the equivalence when the domain spaces are $L$-summands in their biduals. In view of this equivalence, in order to show that the $LBp$ fails for multilinear mappings on $\ell_1 \times \cdots \times \ell_1$, it suffices to see that the $BPBp$ fails.
We slightly modify the counterexample given in \cite{ChoiSong} for bilinear forms, to obtain finite type multilinear mappings that serve as counterexamples to the $BPBp$ for any ideal.
Take $n \in \mathbb N$ and define $T\colon \ell_1 \times \ell_1 \rightarrow \mathbb K$ by
$$
T(x_1,x_2) = \sum_{i,j=1}^{2n^2}x_1(i)x_2(j) (1 - \delta_{ij}) \quad \text{for $\delta_{ij}$ the Kronecker delta.}
$$
Then $\| T\|=1$. Let $\tilde{\x}=(\tilde{x}_1, \tilde{x}_2) \in S_{\ell_1} \times S_{\ell_1}$ where $\tilde{x}_1(i)= \tilde{x}_2(i)=\frac{1}{2n^2}$ for $1 \leq i \leq 2n^2$ and $\tilde{x}_1(i)= \tilde{x}_2(i)= 0$ otherwise and note that $T(\tilde{\x}) = 1 - \frac{1}{2n^2}$. Now, suppose that we can find a norm attaining operator $S \in \mathcal{L}(^2\ell_1 \times \ell_1)$ with $\| S \|=1$ such that
$\| T - S \| < 1$ and take any $\a = (a_1, a_2) \in S_{\ell_1} \times S_{\ell_1}$ with $\vert S(\a) \vert =1$.
Following the calculations in the proof of \cite[Theorem~2]{ChoiSong}, we see that $\| T - S \| < 1$ implies $\|\a - \tilde{\x}\| \geq \frac{1}{2}$. Since $T$ is a finite type bilinear form, this shows that the $BPBp$ fails for any ideal of bilinear forms. Finally, for any Banach space $Y$ and any $y_0 \in Y$ with $\|y_0\|=1$, we can define the finite type $N$-linear mapping $\Phi : \ell_1 \times \cdots \times \ell_1 \rightarrow Y$ by $\Phi(x_1, \dots, x_N) = T(x_1, x_2) e_3'(x_3) \cdots e_N'(x_N) y_0$, where $T$ is defined as above and $(e_i')_{i \in \mathbb N}$ is the dual basic sequence of the canonical vectors.
This gives the desired counterexample to the $BPBp$, and hence to the $LBp$, for any ideal of $N$-linear mappings. To summarize, we have proved the following.

\begin{proposition} \label{contraejemplo LBp en todo ideal}
Let $\left( \U , \| \cdot \|_{\U}\right)$ be a Banach ideal of $N$-linear mappings with $N \geq 2$ and $Y$ any Banach space. Then the $LBp$ fails for $\U(\ell_1 \times \cdots \times \ell_1; Y)$.
\end{proposition}

Now, we give a positive Bishop-Phelps-Bollob\'as-type result for ideals of multilinear mappings. In \cite[Theorem 3.1]{KimLee} the authors prove that if $X$ is uniformly convex then $\mathcal{L}(X;Y)$ has the $BPBp$ for any Banach space $Y$. Analogous results were proved in \cite[Theorem 2.2]{ABGM} for multilinear mappings and in \cite[Theorem 3.1]{Acoet6} for homogeneous polynomials. We adapt the ideas in \cite{Acoet6, KimLee}
to show that a weak version of the $BPBp$ holds for every ideal of multilinear mappings whenever the domain spaces are uniformly convex. We briefly sketch the proof, focusing on the \textit{ideal part}.
Recall that a Banach space $X$ is said to be \textit{uniformly convex} if given $\varepsilon >0$ there exists $0< \delta < 1$ such that $$\text{if $x_1, x_2 \in B_X$ satisfy $\frac{\Vert x_1 + x_2 \Vert}{2} > 1 - \delta$, \quad then \quad $\Vert x_1 - x_2\Vert < \varepsilon$.}$$ In that case, the modulus of convexity of $X$ is given by $$\delta_X(\varepsilon) := \inf \left\{ 1 - \frac{\Vert x_1 + x_2 \Vert}{2}: \, x_1, x_2 \in B_X, \, \Vert x_1 - x_2\Vert > \varepsilon \right\}.$$

Let $(\U, \| \cdot\|_{\U})$ be a Banach ideal of $N$-linear mappings. We say that $\U(\X;Y)$ has the \emph{weak $BPBp$} if for each $\Phi \in \U(\X;Y)$, $\|\Phi\|=1$, and $\varepsilon >0$, there exist $\tilde\beta(\varepsilon,\|\Phi\|_{\U})$ and $\tilde\eta(\varepsilon,\|\Phi\|_{\U})$ \textbf{depending also on $\|\Phi\|_{\U}$} satisfying the inequalities in the definition of the $BPBp$.

Note that if $\U$ is a closed multilinear ideal (i.e.,  $\| \cdot\|_{\U} = \|\cdot\|$) the weak $BPBp$ is just the $BPBp$.

\begin{theorem} \label{BPBp unif convexo}
Let $\left( \U , \| \cdot \|_{\U}\right)$ be a Banach ideal of $N$-linear mappings, $X_1, \dots, X_N$ be uniformly convex Banach spaces and $\X = X_1 \times \cdots \times X_N$. Then $\U(\X;Y)$ has the weak $BPBp$ for every Banach space $Y$.
\end{theorem}

\begin{proof}
Let $\Phi \in \U(\X;Y)$, $\|\Phi\|=1$ and $0<\varepsilon <1$. Consider $\delta(\varepsilon) = \min \{\delta_{X_1}(\varepsilon), \dots, \delta_{X_N}(\varepsilon)\}$ and $\eta(\eps)=\frac{\varepsilon}{2^4}\delta\left(\frac{\varepsilon}{2}\right)$.  Let $\tilde \x \in S_{\X}$ such that
$$
\| \Phi(\tilde \x)\| > 1 -\eta(\eps).
$$
In order to show the result, we define inductively a sequence $\left( (\a_k, \x_k', y_k', \Phi_k)\right)_{k}$ such that $\Phi_k \in \U(\X;Y)$ with $\| \Phi_k \|=1$ and $(\a_k, \x_k' , y_k' ) \in S_{\X} \times S_{\X'} \times S_{Y'}$ satisfying appropriate estimates.

Let $\Phi_1 := \Phi$, $\a_1=\tilde \x = (a_{1,j})_{j=1}^N $ and choose $\x_1' =(x_{1,j}')_{j=1}^N \in S_{\X'}$ and $y_1' \in S_{Y'}$ satisfying
$$
x_{1,j}'(a_{1,j})=1 \quad \text{for all $j=1, \dots,N$} \quad \text{and} \quad \vert y_1'(\Phi_1(\a_1)) \vert> 1 - \eta(\eps).
$$
Suppose that $(\a_k, \x_k', y_k', \Phi_k)$ was defined and satisfies
\begin{equation*}
x_{k,j}'(a_{k,j})=1 \quad \text{for all $j=1, \dots,N$} \quad \text{and} \quad \vert y_k'(\Phi_k(\a_k)) \vert> 1 - \eta\left(\frac{\eps}{2^{k-1}}\right).
\end{equation*}
Consider the auxiliary multilinear function
$$\Psi_{k+1}(\x) := \Phi_k(\x) + \frac{\varepsilon}{2^{k+2}} x_{k,1}'(x_1)\cdots x_{k,N}'(x_N) \Phi_k(\a_k) \quad (\x=(x_j)_{j=1}^N \in \X)$$
which satisfies $1 < \| \Psi_{k+1} \| \leq 1 + \frac{\varepsilon}{2^{k+2}}$. Also, $\Psi_{k+1} \in \U(\X;Y)$ since
both $\Phi_k(\cdot)$ and $x_{k,1}'(\cdot)\cdots x_{k,N}'(\cdot)\Phi_k(\a_k)$ belong to $\U(\X;Y)$.

Now, define $\Phi_{k+1}:= \frac{\Psi_{k+1}}{\| \Psi_{k+1} \|}$ and choose $\a_{k+1}\in S_\X$ and $y_{k+1}'\in S_{Y'}$ such that
\begin{equation*}
\left\vert y_{k+1}'\left(\Psi_{k+1}(\a_{k+1})\right)\right\vert > \| \Psi_{k+1}\| - \eta\left(\frac{\eps}{2^{k}}\right).
\end{equation*}
Up to multiplying the coordinates of $\a_{k+1}$ by modulus 1 complex numbers, we may assume that  $x_{k,j}'(a_{k+1,j}) = |x_{k
,j}'(a_{k+1,j})|$. Finally, choose $\x_{k+1}'$ such that $x_{k+1,j}'(a_{k+1,j}) = 1$ for all $j=1, \dots, N$, completing the $(k+1)$-element of the sequence  $\left( (\a_k, \x_k', y_k', \Phi_k)\right)_k$.

Let us see that $(\Phi_k)_{k}$ is a Cauchy sequence in $\U(\X;Y)$. First observe that $\|\Psi_{k+1}\|_{\U} \leq \|\Phi\|_{\U} + 1$ since $\|\Psi_{k+1}\|_{\U} \leq \|\Phi_k\|_{\U} + \frac{\varepsilon}{2^{k+2}}$ and

$$
\|\Phi_k\|_{\U} = \frac{\|\Psi_{k}\|_{\U}}{\|\Psi_{k}\|} \leq \frac{\|\Phi_{k-1}\|_{\U} + \frac{\varepsilon}{2^{k+1}}}{\|\Psi_{k}\|} \leq \|\Phi_{k-1}\|_{\U} + \frac{\varepsilon}{2^{k+1}}.
$$

On the one hand, we have
\begin{eqnarray} \label{cauchy 1}
\nonumber \|\Phi_{k+1} - \Psi_{k+1}\|_{\U} &=& \left\vert 1 - \| \Psi_{k+1} \| \right\vert \frac{\| \Psi_{k+1} \|_{\U}}{\|\Psi_{k+1}\|} \\
\nonumber &\leq& \left\vert 1 - \| \Psi_{k+1} \| \right\vert \| \Psi_{k+1} \|_{\U} \\
&\leq& \frac{\varepsilon}{2^{k+2}} (\|\Phi\|_{\U} + 1).
\end{eqnarray}
On the other hand,
\begin{equation} \label{cauchy 2}
\| \Psi_{k+1} - \Phi_k \|_{\U} = \left\| \frac{\varepsilon}{2^{k+2}} x_{k,1}'(\cdot)\cdots x_{k,N}'(\cdot) \Phi_k(\a_k) \right\|_{\U} \leq \frac{\varepsilon}{2^{k+2}} \leq \frac{\varepsilon}{2^{k+2}} (\|\Phi\|_{\U} + 1).
\end{equation}

Combining \eqref{cauchy 1} and \eqref{cauchy 2} we obtain
$$
\|\Phi_{k+1} - \Phi_k\|_{\U} \leq \|\Phi_{k+1} - \Psi_{k+1}\|_{\U} + \| \Psi_{k+1} - \Phi_k \|_{\U} \leq \frac{\varepsilon}{2^{k+1}} (\|\Phi\|_{\U} + 1).
$$
Hence, $(\Phi_k)_k$ is a Cauchy sequence in $\U(\X;Y)$ which converges to some $\Phi_\infty \in \U(\X;Y)$ and satisfies $\| \Phi_\infty\|=1$ and $\|\Phi_\infty - \Phi\| < \varepsilon (\|\Phi\|_{\U} + 1).$

As a consequence of the uniform convexity of each $X_j$,  $j=1,\dots, N$, the sequence $(a_{k,j})_k$ is a
Cauchy sequence in $S_{X_j}$ and converges to some element $a_{\infty,j} \in S_{X_j}$ such that $\|
a_{\infty,j} - x_{1,j} \|<\varepsilon$. Now taking $\a_\infty = (a_{\infty,j})_{j=1}^N \in S_{\X}$ we  have
$\| \a_\infty - \tilde \x \|<\varepsilon $. Also, since both sequences $(\Phi_k)_k$ and $(a_{k,j})_k$ are
convergent and $\lim_k \| \Phi_k(\a_k) \| =1$, we see that $\Phi_\infty$ is norm attaining, indeed  $\|
\Phi_\infty(\a_\infty) \|=1$. Then, $\U(\X;Y)$ has the weak $BPBp$ with $\tilde\eta(\varepsilon, \|\Phi\|_{\U})=
\eta(\varepsilon/ (1+\|\Phi\|_{\U}) )$ and  $\tilde\beta(\varepsilon, \|\Phi\|_{\U})= \varepsilon/
(1+\|\Phi\|_{\U}) $.
\end{proof}

We remark that, with a completely analogous proof, the theorem remains valid for any ideal of homogeneous polynomials.


\subsection*{Final remark.}
Related to Proposition \ref{equivalencia BP beta}, in \cite[Proposition 3.3]{Acoet6} it is shown that if $\mathcal{L}(^N\X)$ has the $BPBp$ then $\mathcal{L}(^N\X;Z)$ has the $BPBp$ for every $Z$ with property $(\beta)$. Mimicking the proof of this result, taking care of  \textit{ideal part} as we did in Theorem \ref{BPBp unif convexo}, it can be seen that the corresponding weak statement still holds for any Banach ideal $\left( \U , \| \cdot \|_{\U}\right)$ of $N$-linear mappings.
Also, looking at the proof of Proposition \ref{equivalencia BP beta}, it follows that if $\U(\X;Z)$ has the weak $BPBp$ for some $Z$ with property $(\beta)$, then $\U(\X)$ has the weak $BPBp$. Hence, we have the following.

\begin{proposition}
Let $\left( \U , \| \cdot \|_{\U}\right)$ be a Banach ideal of $N$-linear mappings. Then, $\U(\X)$ has the weak $BPBp$ if and only if $\U(\X;Z)$ has the weak $BPBp$ for every (or some) Banach space $Z$ with property $(\beta)$.
\end{proposition}

\subsection*{Acknowledgements} We wish to thank our friends Domingo Garc\'ia and Manolo
Maestre for  helpful conversations and comments, also for suggesting us to consider
the Lindenstrauss
theorem for symmetric multilinear mappings.


\begin{thebibliography}{99}

\bibitem{Aco98}
M.~D. Acosta.
\newblock On multilinear mappings attaining their norms.
\newblock {Studia Math.},
  131:155--165, 1998.


\bibitem{AAP}
M.~D. Acosta, F.~J. Aguirre, and R. Pay{\'a}.
\newblock There is no bilinear {B}ishop-{P}helps theorem.
\newblock {Israel J. Math.}, 93:221--227, 1996.

\bibitem{AAGM}
M.~D. Acosta, R.~M. Aron, D. Garc\'{\i}a., and Manuel Maestre.
\newblock The {B}ishop-{P}helps-{B}ollob\'as theorem for operators.
\newblock {J. Funct. Anal.}, 254(11):2780--2799, 2008.



\bibitem{Acoet6}
M.~D. Acosta, J. Becerra-Guerrero, Y. S. Choi, D. Garc\'ia, S. K. Kim, H. J. Lee and M. Maestre.
\newblock The {B}ishop-{P}helps-{B}ollob\'as property for bilinear forms
              and polynomials.
\newblock {J. Math. Soc. Japan}, 66(3):957--979, 2014.


\bibitem{ABGKM}
M.~D. Acosta, J. Becerra-Guerrero, D. Garc\'ia, S. K. Kim and M. Maestre.
\newblock Bishop-{P}helps-{B}ollob\'as property for certain spaces of operators.
\newblock {J. Math. Anal. Appl.}, 414(2):532--545, 2014.


\bibitem{ABGM}
M.~D. Acosta, J. Becerra-Guerrero, D. Garc\'ia and M. Maestre.
\newblock The {B}ishop-{P}helps-{B}ollob\'as theorem for bilinear forms.
\newblock {Trans. Amer. Math. Soc.}, 365(11):5911--5932, 2013.


\bibitem{AGM06}
M.~D. Acosta, D. Garc\'{\i}a, M. Maestre.
\newblock A multilinear {L}indenstrauss theorem.
\newblock {J. Funct. Anal.}, 235(1):122--136, 2006.


\bibitem{Arens}
R. Arens.
\newblock The adjoint of a bilinear operation.
\newblock {Proc. Amer. Math. Soc.}, 2:839--848, 1951.

\bibitem{ArCaKo}
R.~M. Aron, B. Cascales and  O. Kozhushkina.
\newblock The {B}ishop-{P}helps-{B}ollob\'as theorem and {A}splund operators.
\newblock {Proc. Amer. Math. Soc.}, 139(10):3553--3560, 2011.


\bibitem{ArGM}
R.~M. Aron, D. Garc\'{\i}a, M. Maestre.
\newblock On norm attaining polynomials.
\newblock {Publ. Res. Inst. Math. Sci.}, 39(1):165--172, 2003.



\bibitem{BishPhel61}
E. Bishop,  R.~R. Phelps.
\newblock A proof that every {B}anach space is subreflexive.
\newblock {Bull. Amer. Math. Soc.}, 67:97--98, 1961.

\bibitem{BishPhel63}
E. Bishop, R.~R. Phelps.
\newblock The support functionals of a convex set.
\newblock {Convexity Proc. Symp. Pure Math. VII},  Amer. Math. Soc., 27--35, 1963.

\bibitem{Boll}
B. Bollob{\'a}s.
\newblock An extension to the theorem of {B}ishop and {P}helps.
\newblock {Bull. London Math. Soc.}, 2:181--182, 1970.

\bibitem{Bombal-PerezGarcia-Villanueva}
F. Bombal, D. P\'erez-Garc\'{\i}a, I. Villanueva.
\newblock Multilinear extensions of Grothendieck's theorem.
\newblock {Q. J. Math.}, 55(4): 441--450, 2004.

\bibitem{BoCaPe11}
G. Botelho, E. \c{C}ali\c{s}kan, D.~M. Pellegrino.
\newblock On the representation of multi-ideals by tensor norms.
\newblock {J. Aust. Math. Soc.}, 90(2): 253--269, 2011.

\bibitem{BoPe05}
G. Botelho, D.~M. Pellegrino.
\newblock Two new properties of ideals of polynomials and applications.
\newblock {Indag. Math.}, 16 (2):157--169, 2005.


\bibitem{Cabello-Garcia-Villanueva}
F. Cabello-S\'anchez, D. P\'erez-Garc\'{\i}a, I. Villanueva.
\newblock Unexpected subspaces of tensor products.
\newblock {J. London Math. Soc.}, 74(2): 512--526, 2006.

\bibitem{CaDiMu07}
D. Carando, V. Dimant, S. Muro.
\newblock Hypercyclic convolution operators on Frechet
spaces of analytic functions.
\newblock {J. Math. Anal. Appl.}, 336(2): 1324--1340, 2007.

\bibitem{CaDiMu09}
D. Carando, V. Dimant, S. Muro.
\newblock Coherent sequences of polynomial ideals on Banach spaces.
\newblock {Math. Nachr.}, 282(8): 1111--1133, 2009.

\bibitem{CaDiMu12}
D. Carando, V. Dimant, S. Muro.
\newblock Holomorphic functions and polynomial ideals on
Banach spaces.
\newblock {Collect. Math.}, 63(1): 71--91, 2012.

\bibitem{CaLaMa12}
D. Carando, S. Lassalle, M. Mazzitelli.
\newblock On the polynomial Lindenstrauss theorem.
\newblock {J. Funct. Anal.}, 263(7): 1809--1824, 2012.

\bibitem{CaMa}
D. Carando,  M. Mazzitelli.
\newblock On bounded holomorphic functions attaining their norms in the bidual.
\newblock {Preprint.}


\bibitem{HandbookI}
P. Casazza.
\newblock Approximation properties.
\newblock In {Handbook of the geometry of {B}anach spaces, {V}ol. {I}},
  pages 271--316. North-Holland, Amsterdam, 2001.

\bibitem{Choi97}
Y.~S. Choi.
\newblock Norm attaining bilinear forms on {$L^1[0,1]$}.
\newblock {J. Math. Anal. Appl.}, 211(1):295--300, 1997.

\bibitem{ChoiKim96}
Y.~S. Choi and S.~G. Kim.
\newblock Norm or numerical radius attaining multilinear mappings and
polynomials.
\newblock {J. London Math. Soc. (2)}, 54(1):135--147, 1996.


\bibitem{ChoiSong}
Y.~S. Choi and H.~G. Song.
\newblock The {B}ishop-{P}helps-{B}ollob\'as theorem fails for bilinear forms on {$l_1\times l_1$}.
\newblock {J. Math. Anal. Appl.}, 360(2):752--753, 2009.

\bibitem{DefFlo93}
A. Defant, K. Floret.
\newblock {Tensor norms and operator ideals}, volume 176 of {
  North-Holland Mathematics Studies}.
\newblock North-Holland Publishing Co., Amsterdam, 1993.


\bibitem{Flo97}
K. Floret.
\newblock Natural norms on symmetric tensor products of normed spaces.
\newblock Proceedings of the Second International Workshop on Functional Analysis (Trier, 1997). Note Mat. 17:
153--188, 1997.

\bibitem{Flo01}
K. Floret.
\newblock Minimal ideals of n-homogeneous polynomials on Banach spaces.
\newblock {Results Math.}, 39(3-4): 201--217, 2001.

\bibitem{FlGa03}
K. Floret, D. Garc\'{\i}a.
\newblock On ideals of polynomials and multilinear mappings
between Banach spaces.
\newblock {Arch. Math.}, 81(3): 300--308, 2003.


\bibitem{JP}
M.~Jim{\'e}nez~Sevilla and R. Pay{\'a}.
\newblock Norm attaining multilinear forms and polynomials on preduals of
{L}orentz sequence spaces.
\newblock {Studia Math.}, 127(2):99--112, 1998.


\bibitem{KimLee}
S. K. Kim and H. J. Lee.
\newblock Uniform convexity and {B}ishop-{P}helps-{B}ollob\'as property.
\newblock {Canad. J. Math.}, 66(2):373--386, 2014.


\bibitem{Lind}
J. Lindenstrauss.
\newblock On operators which attain their norm.
\newblock {Israel J. Math.}, 1:139--148, 1963.


\bibitem{Ma03}
M. Matos.
\newblock Fully absolutely summing and Hilbert-Schmidt multilinear mappings.
\newblock {Collect. Math.}, 54 (2): 111--136, 2003.


\bibitem{tesisMuro}
S. Muro. Funciones holomorfas de tipo acotado e ideales de polinomios homog\'eneos en espacios
de Banach, PhD thesis, Univ. de Buenos Aires, 2010.

\bibitem{P-GV03} D. P\'erez-Garc\'{\i}a; I. Villanueva. \newblock  Multiple summing operators on
Banach spaces.
\newblock {J. Math. Anal. Appl.}, 285 (1):86--96, 2003.
\end{thebibliography}
\end{document}